\documentclass[english]{amsart}

\usepackage{amsthm}
\usepackage{mathdots}
\usepackage{nicematrix}
\usepackage[hidelinks]{hyperref}

\makeatletter
\numberwithin{equation}{section}
\numberwithin{figure}{section}
\theoremstyle{plain}
\newtheorem{thm}{\protect\theoremname}[section]
\theoremstyle{definition}
\newtheorem{defn}[thm]{\protect\definitionname}
\theoremstyle{plain}
\newtheorem{lem}[thm]{\protect\lemmaname}
\theoremstyle{plain}
\newtheorem{prop}[thm]{\protect\propositionname}
\theoremstyle{remark}
\newtheorem*{rem*}{\protect\remarkname}
\theoremstyle{plain}
\newtheorem{cor}[thm]{\protect\corollaryname}
\theoremstyle{remark}
\newtheorem*{acknowledgement*}{\protect\acknowledgementname}

\theoremstyle{plain}
\newtheorem{ex}[thm]
{\protect\examplename}

\usepackage{alex}

\usepackage{tikz-cd}

\newcommand{\mylabel}[2]{#2\def\@currentlabel{#2}\label{#1}}

\makeatother

\usepackage{babel}
\providecommand{\acknowledgementname}{Acknowledgement}
\providecommand{\corollaryname}{Corollary}
\providecommand{\definitionname}{Definition}
\providecommand{\lemmaname}{Lemma}
\providecommand{\propositionname}{Proposition}
\providecommand{\remarkname}{Remark}
\providecommand{\theoremname}{Theorem}
\providecommand{\examplename}{Example}



\begin{document}
\title[Representatives of similarity classes corresponding to ideal classes]{Representatives of similarity classes of matrices over PIDs corresponding
to ideal classes}

\author{Lucy Knight}
\address{Department of Mathematical Sciences, Durham University, Durham, DH1
3LE, UK}
\curraddr{Department of Mathematics, Dartmouth College, Hanover, NH 03755, USA}
\email{lucy.c.knight.gr@dartmouth.edu}

\author{Alexander Stasinski}
\address{Department of Mathematical Sciences, Durham University, Durham, DH1
3LE, UK}
\email{alexander.stasinski@durham.ac.uk}

\begin{abstract}
For a principal ideal domain $A$, the Latimer--MacDuffee correspondence
sets up a bijection between the similarity classes of matrices in
$\M_{n}(A)$ with irreducible characteristic polynomial $f(x)$ and
the ideal classes of the order $A[x]/(f(x))$. We prove that when
$A[x]/(f(x))$ is maximal (i.e., integrally closed, i.e., a Dedekind
domain), then every similarity class contains a representative that
is, in a sense, close to being a companion matrix. The first step
in the proof is to show that any similarity class corresponding to
an ideal (not necessarily prime) of degree one contains a representative
of the desired form. The second step is a previously unpublished result due to Lenstra that
implies that when $A[x]/(f(x))$ is maximal, every ideal class contains
an ideal of degree one.
\end{abstract}

\maketitle

\section{Introduction}

Latimer and MacDuffee \cite{Latimer-MacDuffee} showed that there
is a bijection between the similarity classes of matrices with integer
entries (i.e., the orbits under the action of $\GL_{n}(\Z)$ on $\M_{n}(\Z)$
by conjugation) with irreducible characteristic polynomial $f(x)\in\Z[x]$
and the ideal classes of the order $\Z[x]/(f(x))$. Another proof
of this was given by Taussky \cite{Taussky-Latimer-MacDuffee}. Later
generalisations show that the correspondence also holds when $\Z$
is replaced by an arbitrary principal ideal domain (PID); see Section~\ref{sec:The-Latimer=002013MacDuffee-correspond}.

In \cite{Ochoa-I}, Ochoa claimed that every similarity class with
characteristic polynomial $f(x)$ contains a representative of the
form
\[
C_{f}(a,z)=\begin{pNiceMatrix}[xdots/shorten=0.4cm]
0 &  1 & 0 & \Cdots  & 0\\
\vdots & \Ddots &\Ddots & \Ddots& \vdots\\
0 & \Cdots &0 &1 & 0\\
u_{n-1} & u_{n-2} & \Cdots\Cdots & u_{1} & -f(z)a^{-1}\\
a & 0 & \Cdots \Cdots & 0 & z
\end{pNiceMatrix},
\]
where $a,z\in \Z$ are such that $a$ divides $f(z)$ and the $u_i\in \Z$ are explicitly determined by $f(x)$ and $z$ (see Lemma~\ref{lem:C_f(a,z)}).
For the case of $n=2$,
\[
C_{f}(a,z)=\begin{pmatrix}
    u_1 & -f(z)a^{-1}\\ a & z
\end{pmatrix}
\]
is a general $2\times 2$ matrix over $\Z$, that is, any $2\times 2$ matrix is already of the form $C_{f}(a,z)$ for some $a,z\in \Z$, so it is trivial that every similarity class contains such a representative.
Note that since $f(x)$ is irreducible, $C_{f}(a,z)$ is $\GL_{n}(\Q)$-similar to the companion
matrix of $f(x)$ and that if $f(z)a^{-1}=\pm1$, then $C_{f}(a,z)$
is $\GL_{n}(\Z)$-similar to the companion matrix of $f(x)$. 
In general, it is often of great value to have a normal form for matrices under similarity. 
Ochoa's motivation was to give an explicit version of the Latimer--MacDuffee
correspondence and transfer the structure of the ideal class monoid
of $\Z[x]/(f(x))$ to sets of matrices of the form $C_{f}(a,z)$, up to similarity. 
Unfortunately, \cite{Ochoa-I} contains several mistakes (e.g., Theorem~I-4
is not true) and while Ochoa's follow-up paper \cite{Ochoa-II} was
an attempt to rectify some of these mistakes, this was ultimately inconclusive, as already the central Lemma I-1 in \cite{Ochoa-II} has a counter-example (see Rehm's review \cite{Rehm-review-of-Ochoa-II} of \cite{Ochoa-II}).

It seems that the approach in \cite{Ochoa-I} and \cite{Ochoa-II}
is not salvageable. On the other hand, Rehm \cite{Rehm-Ochoa} did
establish Ochoa's claim that every similarity class as above contains
a representative of the form $C_{f}(a,z)$ under the (necessary) condition
that $\Z[x]/(f(x))$ is the maximal order (i.e., the ring of integers)
of the number field $\Q[x]/(f(x))$. (Note that Rehm used conventions
that led to the transpose of $C_{f}(a,z)$.) 

The purpose of this paper is to generalise Rehm's result to the case
where $\Z$ is replaced by an arbitrary PID. Our main result is Theorem~\ref{thm:main}. The most serious obstacle
to such a generalisation is Rehm's use of the fact that every ideal
class in the ring of integers of a number field contains a prime ideal
of degree one. The classical proofs of this fact are rather non-trivial,
based either on the Chebotarev density theorem or on analytic properties
of the Dedekind zeta function, and it is not clear whether an analogous
proof exists in the more algebraic setting of an arbitrary PID. Nevertheless,
Rehm showed more generally (i.e., when $\Z[x]/(f(x))$ is not necessarily
maximal) that any similarity class corresponding to an ideal class
containing an ideal of the form $\mfq_{1}\cdots\mfq_{r}$, where $r\in\N$,
each $\mfq_{i}$ is a prime ideal of degree one and the $\mfq_{i}$
lie over pairwise distinct primes, contains a matrix $\cC_{f}(a,z)$
(see \cite[Theorem~2]{Rehm-Ochoa}; note again that Rehm used the
transposed matrix). Rehm's proof of this fact goes through without difficulty for
the order $A[x]/(f(x))$, where $A$ is an arbitrary PID and $f(x)\in A[x]$
is irreducible. It is however not clear (and perhaps not even true)
that if $A[x]/(f(x))$ is assumed to be the maximal order, hence a
Dedekind domain, then each of its ideal classes contains an ideal
of the form $\mfq_{1}\cdots\mfq_{r}$, as above.

Generalising \cite[Theorem~2]{Rehm-Ochoa}, we show that any similarity
class in $\M_{n}(A)$ corresponding to an ideal class of $A[x]/(f(x))$
containing an ideal of the form
\[
\mfq_{1}^{e_{1}}\cdots\mfq_{r}^{e_{r}},
\]
where $r\in\N\cup\{0\}$, $e_{i}\in\N$, each $\mfq_{i}$ is a prime ideal of degree
one and of ramification index one, and the $\mfq_{i}$ lie over pairwise
distinct primes, contains a matrix $\cC_{f}(a,z)$ (note that when $r=0$ we have the empty product $(1)$, which lies in the trivial ideal class). The main point
here is that, unlike Rehm, we allow higher powers of our prime ideals. We prove
this by working with ``ideals (not necessarily prime) of degree one''.
If $A$ is a Dedekind domain with field of fractions $K$, $L/K$
is a finite separable extension and $B$ is the integral closure of
$A$ in $L$, then it is standard to say that a prime ideal $\mfq$
of $B$ is of degree one if $[B/\mfq:A/(A\cap\mfq)]=1$. As a generalisation
of this we say that a proper ideal $\mfb$ of $B$ is of \emph{degree
one} if $A+\mfb=B$, that is, if the inclusion of $A/(A\cap\mfb)$
in $B/\mfb$ is an isomorphism. It turns out that ideals of the form
$\mfq_{1}^{e_{1}}\cdots\mfq_{r}^{e_{r}}$, as above, are of degree
one. In fact, we prove that any similarity class corresponding to
an ideal class containing an ideal of degree one contains a matrix
$C_{f}(a,z)$.

To prove that every similarity class contains a matrix $C_{f}(a,z)$
when $A[x]/(f(x))$ is the maximal order, it remains to show that
in this case, every ideal class contains an ideal of the form $\mfq_{1}^{e_{1}}\cdots\mfq_{r}^{e_{r}}$,
as above. This is achieved by a result of Lenstra (Theorem~\ref{thm:Lenstra}), which
is a strengthening of a previous result of Lenstra--Stevenhagen \cite[Theorem~1]{Lenstra-Stevenhagen} with
the extra condition that the prime ideals lie over pairwise
distinct primes. 
The proof of Theorem~\ref{thm:Lenstra} was previously unpublished and kindly communicated to us by Lenstra. 

\section{\label{sec:Ideals-of-degree-one}Ideals of degree one}

Let $A$ be a Dedekind domain with field of fractions $K$, let $L/K$
be a finite separable field extension of degree $n$ and let $B$
be the integral closure of $A$ in $L$. Then $B$ is a Dedekind domain
(see \cite[I.6.2]{Lorenzini}). An \emph{$A$-order in $L$} is a
subring $\cO$ of $L$ that is a finitely generated $A$-module such
that $K\cO=L$. In particular, $B$ is an $A$-order in $L$ (\cite[I.4.7]{Lorenzini}).
Moreover, every element of an $A$-order $\cO$ is integral over $A$
(see \cite[(8.6)]{Reiner-max-ord}), so $\cO\subseteq B$. Let $\theta\in B$
be such that $L=K(\theta)$; then $A[\theta]$ is an $A$-order in
$L$. 
\begin{defn}
Let $\cO$ be an $A$-order in $L$. A proper ideal $\mfb$ of $\cO$
is said to be of \emph{degree one (over $A$)} if $A+\mfb=\cO$.
\end{defn}

By the second isomorphism theorem, $(A+\mfb)/\mfb\cong A/(A\cap\mfb)$,
so $\mfb$ is of degree one if and only if the map $A/(A\cap\mfb)\rightarrow\cO/\mfb$
induced by the embedding $A\rightarrow\cO$ is an isomorphism. Thus
the notion of an ideal of degree one generalises the standard notion
of a prime ideal of degree one.
\begin{lem}
\label{lem:if z iso and ideal gen}Let $\mfb$ be an ideal of $A[\theta]$
and $\mfa=A\cap\mfb$. Suppose that there exists a $z\in A$ such
that $\theta-z\in\mfb$. Then the embedding $A\rightarrow A[\theta]$
induces an isomorphism $A/\mfa\rightarrow A[\theta]/(\theta-z,\mfa)$
and $\mfb=(\theta-z,\mfa)$.
\end{lem}

\begin{proof}
We have $\mfb\supseteq(\theta-z,\mfa)$, so 
\[
\mfa=A\cap\mfb\supseteq A\cap(\theta-z,\mfa)\supseteq\mfa
\]
and thus $\mfa=A\cap(\theta-z,\mfa)$. This implies that the map $A/\mfa\rightarrow A[\theta]/(\theta-z,\mfa)$
is injective. The map is also surjective, as for any $g(x)\in A[x]$,
we have $g(\theta)-g(z)\in(\theta-z,\mfa)$, so $g(z)+\mfa\in A/\mfa$
maps to $g(\theta)+(\theta-z,\mfa)\in A[\theta]/(\theta-z,\mfa)$.
We thus have a commutative diagram$$
\begin{tikzcd}[column sep=tiny] 
A[\theta]/(\theta-z,\mfa)\arrow{rr} & {} & A[\theta]/\mfb\\
{} & A/\mfa \arrow["\cong"]{ul}\arrow[swap,"\cong"]{ur} & {}
\end{tikzcd}
$$implying that the top surjective map must be an isomorphism, that
is, 
\[
\mfb=(\theta-z,\mfa). \qedhere
\]
\end{proof}
\begin{lem}
\label{lem:deg1-equivalence}Let $\mfb$ be a proper ideal of $A[\theta]$
and $\mfa=A\cap\mfb$. Then $\mfb$ is of degree one if and only if
there exists a $z\in A$ such that $\mfb=(\theta-z,\mfa)$. 
\end{lem}

\begin{proof}
Assume that $\mfb$ is of degree one and choose $z\in A$ such that
$\theta\in z+\mfb$.
Since $\theta-z\in\mfb$,
Lemma~\ref{lem:if z iso and ideal gen} implies that $\mfb=(\theta-z,\mfa).$

Conversely, assume that there is a $z\in A$ such that $\mfb=(\theta-z,\mfa)$. 
By Lemma~\ref{lem:if z iso and ideal gen}
$A/\mfa\rightarrow A[\theta]/(\theta-z,\mfa)=A[\theta]/\mfb$ is an
isomorphism. Thus $\mfb$ is of degree one, as it is a proper ideal.
\end{proof}
Recall that a prime ideal $\mfq$ of $B$ has ramification index one
(over $A)$ if it appears with exponent $1$ in the prime ideal decomposition
of $(A\cap\mfq)B$ and it is called \emph{unramified} if, in addition,
the field extension $(B/\mfq)/(A/(A\cap\mfq))$ is separable (see
\cite[Chapter~I, (8.3)]{Neukirch}).
\begin{prop}
\label{prop:deg-one-characterisation}If $\mfb$ is an ideal of $B$
of degree one, then each of its prime ideal factors is of degree one.
Moreover, if $\mfq_{1},\dots,\mfq_{r}$ are prime ideals of $B$ lying
over distinct primes in $A$ (i.e., $A\cap\mfq_{i}\neq A\cap\mfq_{j}$
for $i\neq j$), and each $\mfq_{i}$ is of degree one and has ramification
index one, then $\mfq_{1}^{e_{1}}\cdots\mfq_{r}^{e_{r}}$ is of degree
one for any $e_{i}\in\N$.
\end{prop}

\begin{proof}
Assume that $\mfb$ is of degree one, let $\mfq$ be a prime ideal
factor and write $\mfb=\mfq\mfb'$, for some ideal $\mfb'$. Since
$\mfb$ is of degree one, we have $A+\mfq\mfb'=B$ and thus (as $\mfq\mfb'\subseteq\mfq$)
\[
A+\mfq=B,
\]
that is, $\mfq$ is of degree one. This proves the first assertion.

For the second assertion it is enough to prove the following two statements:
\begin{enumerate}
\item \label{enu:1-proof}If $\mfq$ is a prime ideal of $B$ of degree
one and ramification index one, then $\mfq^{e}$ is of degree one,
for all $e\in\N$.
\item \label{enu:2-proof}If $\mfb$ and $\mfc$ are two ideals of $B$
of degree one such that $A\cap\mfb$ is coprime to $A\cap\mfc$, then
$\mfb\mfc$ is of degree one.
\end{enumerate}
To prove (\ref{enu:1-proof}), let $\mfp=A\cap\mfq$. Then $\mfp^{e}\subseteq A\cap\mfq^{e}$,
so $A\cap\mfq^{e}=\mfp^{m}$, for some $m\leq e$. Further, we have
\[
(\mfp B)^{m}=\mfp^{m}B\subseteq\mfq^{e},
\]
and since $\mfq$ has ramification index one, it appears with multiplicity
$m$ in $(\mfp B)^{m}$, hence $e\leq m$, forcing $m=e$, so that
\[
A\cap\mfq^{e}=\mfp^{e}.
\]
In order to to prove that $\mfq^{e}$ is of degree one we thus need
to show that $A/\mfp^{e}\rightarrow B/\mfq^{e}$ is surjective (i.e.,
an isomorphism). We do this by induction on $e$, the case $e=1$
holding because of the hypothesis that $\mfq$ is of degree one. Assume
that the map is surjective for some $i\geq1$ and consider the commutative
diagram$$
\begin{tikzcd}
1\arrow{r} & \mfq^i/\mfq^{i+1}\arrow{r} & B/\mfq^{i+1}\arrow{r} & B/\mfq^i\arrow{r} & 1\\
1\arrow{r} & \mfp^i/\mfp^{i+1}\arrow{r}\arrow[hook]{u} & A/\mfp^{i+1}\arrow{r}\arrow[hook]{u} & A/\mfp^i\arrow{r}\arrow["\cong"]{u} & 1
\end{tikzcd}
$$If the leftmost vertical inclusion is surjective then the middle one
is as well and we will be done. But $\mfp^{i}/\mfp^{i+1}$
is a vector space over $A/\mfp$ of dimension $1$ (since $A$
is of dimension $1$) and similarly $\mfq^{i}/\mfq^{i+1}$ is a vector space
over $B/\mfq$ of dimension $1$. But $A/\mfp\rightarrow B/\mfq$
is an isomorphism, as $\mfq$ is of degree one, and this isomorphism
is compatible with the leftmost vertical inclusion, which is therefore
an injective linear map of $A/\mfp$-vector spaces of dimension $1$ and
is hence surjective. Thus $\mfq^{e}$ has degree one for all $e\in\N$.

We now prove (\ref{enu:2-proof}). Since $A\cap\mfb$ is assumed to
be coprime to $A\cap\mfc$ we have that $\mfb$ is coprime to $\mfc$.
By the Chinese Remainder Theorem, we thus have
\[
B/\mfb\mfc\cong B/\mfb\times B/\mfc,
\]
as well as the fact that products of coprime ideals are intersections, so that
\[
(A\cap\mfb)(A\cap\mfc)=(A\cap\mfb)\cap (A\cap\mfc)=A\cap \mfb \cap \mfc = A\cap \mfb\mfc
\]
and
\[ 
A/(A\cap\mfb\mfc) \cong A/(A\cap\mfb)\times A/(A\cap\mfc).
\]

We thus have a commutative diagram
$$
\begin{tikzcd} 
B/\mfb\mfc\arrow["\cong"]{r} & B/\mfb\times B/\mfc\\
A/(A\cap\mfb\mfc)\arrow["\cong~"]{r}\arrow[hook]{u} & A/(A\cap\mfb)\times A/(A\cap\mfc) \arrow[hook]{u}
\end{tikzcd},
$$
where the rightmost inclusion is given by the inclusions in each component.
This diagram immediately implies that $\mfb\mfc$ is of degree one,
since $\mfb$ and $\mfc$ are. 
\end{proof}
\begin{rem*}
It is not true that every prime ideal factor of an ideal of $B$ of
degree one has ramification index one. For example, any ramified prime
ideal of a quadratic number field has degree one (over $\Z$).
\end{rem*}

\subsection{Every ideal class contains an ideal of degree one}

This section contains a proof of the following result of Lenstra.
Let $S$ be a finite set of prime ideals of $B$.
Then every ideal class of $B$ contains an ideal of the form
\[
\mfq_{1}^{e_{1}}\cdots\mfq_{r}^{e_{r}},
\]
where $r\in\N\cup \{0\}$, $e_{i}\in\N$, each $\mfq_{i}$ is a prime ideal of $B$ of
degree one such that $\mfq_{i}\not\in S$ and the $\mfq_{i}$ lie
over pairwise distinct prime ideals of $A$, that is, $A\cap\mfq_{i}\neq A\cap\mfq_{j}$
unless $i=j$. 
Note that we allow the empty product $(1)=B$ when $r=0$, so the trivial ideal class is covered. 
As is easily seen, this result is equivalent to Theorem~\ref{thm:Lenstra}, but in the proof it is more convenient to use the formulation of the latter. 
The proof builds on and strengthens a previous result of Lenstra and Stevenhagen \cite[Theorem~1]{Lenstra-Stevenhagen}. The following two lemmas on which Theorem~\ref{thm:Lenstra} depends are also due to Lenstra.

If $A$ is a field, then $B$ is as well so the class group of $B$ is trivial and Theorem~\ref{thm:Lenstra} likewise. Then to prove the theorem, it will be sufficient to assume for the rest of this subsection that $A$ is not a field, in which case it is infinite.

Note that we do not require $A/\mfp$ to be finite in any of the statements below. 

\begin{lem}
\label{lem:Lenstra1}
For each positive integer $m$, the number of primes $\mfp$ of $A$ for which $|A/\mfp|$ is at most $m$ is finite.
\end{lem}

\begin{proof}
Consider the polynomial $F=\prod_{q}(X^{q}-X)\in K[X]$, with $q$ ranging over the set of prime powers ($>1$) that are at most $m$. Since $F$ is non-zero, it has only a finite number of zeros in $A$, so we can choose $a\in A$ with $F(a)$ non-zero. If $\mfp$ is a prime of $A$ with $|A/\mfp|$ at most $m$, then with $|A/\mfp|=q$ we have $a^{q}\equiv a\mod\mfp$, so $\mfp$ divides $F(a)$. But $F(a)$ has only finitely many prime divisors, which proves the lemma. 
\end{proof}

In the following lemma, note that there is always a non-zero $d\in A$ such that $dB\subset A[\theta]$ (e.g., the discriminant of the minimal polynomial of $\theta$ over $K$).
\begin{lem}
\label{lem:Lenstra2}
Let $d$ be a non-zero element of $A$ such that $dB\subset A[\theta]$ and such that every prime $\mfp$ of $A$ for which $|A/\mfp|$ is at most $[L:K]$ divides $dA$. Let $U$ be a finite set of primes of $B$ not dividing $dB$, and suppose that $\mfq$ is a prime of degree one of $B$ that does not divide $dB$. Then there exists a non-zero element $x\in B$ such that:
\begin{enumerate}
    \item \label{enu:Lenstra2-1} $x\equiv 1\mod dB$;
    \item \label{enu:Lenstra2-2} one has $$xB = \mfq\prod_{i = 1}^t \mfq_i^{e_i},$$ with $t\in \N\cup\{0\}$, $e_i\in \N$, where the $\mfq_1,\dots,\mfq_t$ are primes of $B$ satisfying the following properties.
    \begin{enumerate}
     \item[\emph{i)}] $\mfq_i \neq \mfq$  and $\mfq_i$ does not divide $dB$;
        \item[\emph{ii)}] $\mfq_i$ is of degree one and the primes $A \cap \mfq_i$ and $A \cap \mfq_j$ are different whenever $i\neq j$;
        \item[\emph{iii)}]  $\mfq_i \not \in U$.
       
    \end{enumerate}
\end{enumerate}
\end{lem}

\begin{proof}

Let $\mfp=A \cap \mfq$ and $\beta=d\theta$. 
Then $B_{\mfp}=A_{\mfp}[\beta]$, and from Kummer--Dedekind's theorem we obtain $\mfq=\mfp B+(\beta+u)B$ for some $u\in A$. Moreover, we may assume that $\beta+u\not \in \mfq^2$ (if $\mfp\subseteq \mfq^2$, $\beta+u\in \mfq^2$ would lead to $\mfq \subseteq \mfq^2$ and if $\mfp\not \subseteq \mfq^2$ and $\beta+u \in \mfq^2$, then $\beta+u+u'\not \in \mfq^2$ for any $u'\in \mfp\setminus \mfq^2$).

The idea is to replace $u$ by an element $v\in A$ satisfying the conditions \ref{enu:v-a}-\ref{enu:v-c} below and show that $x=\beta+v$ satisfies the statements of the lemma. In order to state condition \ref{enu:v-c}, we need to introduce a bit of notation. 
Let $\mfr$ be a prime of $A$ lying below a prime in $U$. Then $|A/\mfr|>[L:K]$ by hypothesis. There are at most $[L:K]$ primes of $B$ of degree one over $\mfr$, so say $m$ such primes exist. By Kummer--Dedekind's theorem, each prime $\mfs$ of $U$ above $\mfr$ has the form $\mfr B+(\beta+r_{\mfs})B$ for some $r_{\mfs}\in A$ that is uniquely determined modulo $\mfr$. Since $|A/\mfr|>[L:K]\geq m$, we can choose $a_{\mfr}\in A$ such that $a_{\mfr}$ is not congruent to any of the $r_{\mfs}$ modulo $\mfr$.

Let $U_A=\{A\cap\mfs \mid \mfs \in U\}$ and consider the following three conditions.
\begin{enumerate}
    \item[\mylabel{enu:v-a}{(a)}] $v\equiv 1 \mod dA$,
    \item[\mylabel{enu:v-b}{(b)}]  $v\equiv u \mod \mfq^2$,
    \item[\mylabel{enu:v-c}{(c)}]  $v\equiv a_{\mfr} \mod \mfr$ for all $\mathfrak{r} \in U_A\setminus\{\mfp\}$.
\end{enumerate}
We claim that by the Chinese Remainder Theorem for the ring $A$, there exists an element $v\in A$ that satisfies the three conditions \ref{enu:v-a}-\ref{enu:v-c}. For this, we only need to verify that the ideals appearing in the moduli in the three conditions are pairwise coprime. Indeed, $u,v\in A$ so \ref{enu:v-b} is equivalent to $v\equiv u\mod A\cap\mfq^2$ and $dA$ and $A\cap \mfq^2$ are coprime since $A\cap \mfq^2$ is a power of $\mfp$ and $\mfq$ does not divide $dB$. Moreover, each $\mfr$ in \ref{enu:v-c} is different from $\mfp$ and each $\mfr\in U_A$ is coprime to $dA$ because this is so for $A\cap\mfs$ for any $\mfs\in U$, by hypothesis.

Let $x:=\beta+v$, for $v$ satisfying \ref{enu:v-a}-\ref{enu:v-c}. Condition \ref{enu:v-b} ensures that $x$ is non-zero since we have chosen $u$ such that $\beta+u\not \in \mfq^2$.
Condition \ref{enu:v-a} ensures (\ref{enu:Lenstra2-1}) since $\beta\in dB$.

Furthermore, condition \ref{enu:v-b} ensures $x\in \mfq\setminus \mfq^2$, so that $\mfq$ occurs exactly once in the prime factorisation of $xB$. Thus, this prime factorisation has the form 
$$xB=\mfq\prod_{i=1}^{t} \mfq_i^{e_i}$$
as in (\ref{enu:Lenstra2-2}), for $\mfq_i$ pairwise distinct primes of $B$ and it follows from (\ref{enu:Lenstra2-1}) that none of the $\mfq_i$ divide $dB$. This proves (\ref{enu:Lenstra2-2})\,\emph{i)}.

Next, we prove (\ref{enu:Lenstra2-2})\,\emph{ii)}. Write $\mfp_i = A\cap \mfq_i$. From $\mfp_i$ not dividing $d$ it follows that $B_{\mfp_i}=A_{\mfp_i}[\beta]$ so $B/(\mfp_i B+(\beta+v)B) \hookrightarrow A_{\mfp_i}/\mfp_iA_{\mfp_i}\cong A/\mfp_i$ has $A/\mfp_i$-dimension at most 1. However, $\mfq_i$ contains $\mfp_iB+(\beta+v)B$ and because of the $A/\mfp_i$-surjection $B/(\mfp_i B+(\beta+v)B) \rightarrow B/\mfq_i$ this dimension is at least $[B/\mfq_i: A/\mfp_i]$, so we obtain both that 
$$\mfq_i=\mfp_i B+(\beta+v)B$$ 
and $[B/\mfq_i: A/\mfp_i]=1$. Therefore each $\mfq_i$ has degree one over $A$, and $\mfp_i=\mfp_j$ implies $\mfq_i=\mfq_j$, which implies $i=j$. 
Similarly, each $\mfp_i$ is distinct from $\mfp$ because if $\mfp_i = 
\mfp$, then $\mfq_i = \mfp B+(\beta+v)B=\mfq$; contradiction. 

It remains to prove (\ref{enu:Lenstra2-2})\,\emph{iii)}, that is, that none of the $\mfq_i$ are in $U$. If $\mfq_i\in U$, then $\mfp_i \in U_A\setminus\{\mfp\}$ (since we have shown that $\mfp_i \neq \mfp$), so \ref{enu:v-c} implies that $v\equiv a_{\mfp_i} \mod \mfp_i$.
Thus, as $x=\beta+v\in \mfq_i$, we obtain $\beta+a_{\mfp_i}\in \mfq_i = \mfp_i B+(\beta+r_{\mfq_i})B$. As $\beta+r_{\mfq_i} \in \mfq_i$, we get $a_{\mfp_i}-r_{\mfq_i}\in\mfq_i$. But $a_{\mfp_i},r_{\mfq_i}\in A$, so $a_{\mfp_i}-r_{\mfq_i}\in A \cap \mfq_i=\mfp_i$, contradicting $a_{\mfp_i}\not\equiv r_{\mfq_i}\mod \mfp_i$. \qedhere 
\end{proof}

\begin{thm}
\label{thm:Lenstra} Let $S$ be a finite set of primes of $B$. For every ideal class $c$ of $B$, there exist  a finite set $T$ (possibly empty) of primes of degree one of $B$ and a function $e:T \to \N$ such that $c$ contains an ideal of the form $\prod_{\mfq \in T} \mfq^{e(\mfq)}$, where 
\begin{enumerate}
    \item \label{enu:LenstraTheorem1} $T$ is disjoint from $S$,
    \item \label{enu:LenstraTheorem2} For any two distinct $\mfq, \mfq' \in T$ the primes $A \cap \mfq$ and $A \cap \mfq'$ are different.
\end{enumerate} 
\end{thm}

Note that in \cite{Lenstra-Stevenhagen}, the lemma is invoked for a prime $\mfq$, but the condition of the lemma that $\mfq$ does not divide $dB$ is disregarded. To repair this, one should introduce $d$ earlier and add its prime divisors to $S$, as we do below.

\begin{proof}
First, we enlarge $S$ in two ways. By Lemma~\ref{lem:Lenstra1}, for only finitely many primes $\mfp$ of $A$ is $|A/\mfp|$ at most $[L:K]$, and the finitely many primes of $B$ lying over such $\mfp$ are adjoined to $S$. Secondly, choose  $d$ as in Lemma~\ref{lem:Lenstra2}, but such that $d$ is divisible by all primes in $S$, and replace $S$ by the set of prime divisors of $dB$ in $B$. By \cite[Theorem~1]{Lenstra-Stevenhagen}, the class group of $B$ is generated by the primes of degree one that are not in $S$. Hence, the ideal class $c$ contains a fractional ideal of the form $\prod_{\mfq \in T} \mfq^{e(\mfq)}$, where $T$ is a finite set of primes of degree one of $B$, the set $T$ is disjoint from $S$, and $e: T \to \Z\setminus\{0\}$ is a map. 

What remains to be done, is achieve (\ref{enu:LenstraTheorem2}), as well as to achieve that all $e(\mfq)$ are positive. We deal with these in order, first achieving (\ref{enu:LenstraTheorem2}), which requires that all $A \cap \mfq$, with $\mfq \in T$, are pairwise different. This we do by induction on the number of “coincidences" in the product, a coincidence being defined as a prime $\mfq \in T$ for which there exists $\mfq' \in T$, $\mfq'\neq \mfq$, with $A \cap \mfq = A \cap \mfq'$. If there are no such $\mfq$, then (\ref{enu:LenstraTheorem2}) is achieved. Next suppose that $\mfq^{*}$ is a coincidence. Apply Lemma~\ref{lem:Lenstra2} to $\mfq = \mfq^{*}$ and $U$ equal to the set of primes of $B$ that lie over the same prime of $A$ as some prime in $T$. Then we find, with the notation of the conclusion of Lemma~\ref{lem:Lenstra2}, that $\mfq^{*}$ is in the same ideal class as $\prod_{i = 1}^t \mfq_i^{-e_i}$ for certain primes $\mfq_i$ of degree one of $B$ not in $S$ nor in $U$, with all $A \cap \mfq_i$ pairwise different. Thus, one can replace the factor $(\mfq^{*})^{e(\mfq^{*})}$ in our product by $\prod_{i = 1}^t \mfq_i^{-e_ie(\mfq^{*})}$. By the choice of $U$, this replacement does not introduce new coincidences, and it eliminates the coincidence $\mfq^{*}$.  Hence in the new product the number of coincidences is strictly smaller than in the old product.

We have now proved that every ideal class of $B$ contains a fractional ideal of the form $\prod_{\mfq \in T} \mfq^{e(\mfq)}$, where $T$ is a finite set of primes of degree one of $B$, the set $T$ is disjoint from $S$, all primes $A \cap \mfq$ of $A$ are pairwise different as $\mfq$ ranges over $T$, and $e: T \to \Z\setminus \{0\}$ is a map. It remains to change this product such that all $e(\mfq)$ are positive. This we do by induction on the number of $\mfq \in T$ for which $e(\mfq) < 0$. If this number is 0, we are done. Suppose this number is positive, and let $\mfq^{*}\in T$ be such that $e(\mfq^{*}) < 0$. We again apply Lemma~\ref{lem:Lenstra2} to the same choices of $\mfq$ and $U$, which again enables us to replace the factor $(\mfq^{*})^{e(\mfq^{*})}$ in our product by $\prod_{i = 1}^t \mfq_i^{-e_ie(\mfq^{*})}$; note that all exponents $-e_ie(\mfq^{*})$ are positive. By construction, this new product satisfies the same conditions as we listed for the original product, but the number of $\mfq \in T$ for which $e(\mfq) < 0$ has decreased by 1. Thus, the proof is now finished by induction.
\end{proof}

Note that the difference between the above theorem and the results
in \cite{Lenstra-Stevenhagen} is the extra condition that the $\mfq_{i}$
lie over distinct prime ideals of $A$. This is necessary for us because from this we can deduce:
\begin{cor}
\label{cor:every ideal class contains deg one} Every ideal class
of $B$ contains an ideal of degree one.
\end{cor}

\begin{proof}
Let $S$ be the set of prime ideals of $B$ that are not unramified.
Then $S$ is finite (see \cite[Chapter~I, (8.4)]{Neukirch}) so by
Theorem~\ref{thm:Lenstra}, every ideal class of $B$ contains an
ideal $\mfb$ that is a product of unramified prime ideals of degree
one lying above pairwise distinct prime ideals of $A$. Since unramified
primes have ramification index one, Proposition~\ref{prop:deg-one-characterisation}
implies that $\mfb$ is of degree one. 
\end{proof}

\section{\label{sec:The-Latimer=002013MacDuffee-correspond}The Latimer--MacDuffee
correspondence and ideal matrices}

We keep the $A,K,B,L$ set-up from before, with $\theta\in B$ such that $L=K(\theta)$ and $n=[L:K]$. From now on and throughout
the rest of the paper we assume that $A$ is PID. The following is
a well-known fact that depends on this assumption on $A$.
\begin{lem}
\label{lem:Every ideal free rank n} Let $\cO$ be an $A$-order in
$L$ and $\mfb$ a non-zero ideal of $\cO$. Then $\mfb$ is free
of rank $n$ over $A$. 
\end{lem}
Let 
\[
T:A[\theta]\longrightarrow \M_n(A)
\]
be the regular embedding with respect to the basis $1,\theta,\dots,\theta^{n-1}$, that is, for any $x\in A[\theta]$, 
\[
T(x)\begin{pmatrix}1\\
\vdots\\
\theta^{n-1}
\end{pmatrix}=x\begin{pmatrix}1\\
\vdots\\
\theta^{n-1}
\end{pmatrix}.
\]
In particular, if 
\[
f(x)=x^{n}+k_{n-1}x^{n-1}+\dots+k_{1}x+k_{0}\in A[x]
\]
denotes the minimal polynomial of $\theta$ over $K$, then
\[
T(\theta)=\begin{pNiceMatrix}[xdots/shorten=0.4cm]
0 & 1 & 0 & \Cdots & 0\\
\Vdots & \Ddots & \Ddots & \Ddots & \Vdots\\
\Vdots &  & \Ddots & \Ddots & 0\\
0 & \Cdots & \Cdots & 0 & 1\\
-k_{0} & -k_{1} &  & \Cdots & -k_{n-1}
\end{pNiceMatrix}.
\]
Let $I$
be the kernel of the surjective homomorphism $A[x]\rightarrow A[\theta]$,
$x\mapsto\theta$. Even though $A[x]$ is not necessarily a PID, it is true that $I$ is principal. Indeed, let $g(x)\in I$ be a non-zero element. Then,
as $f(x)$ is the minimal polynomial of $\theta$ over $K$, we have $g(x)=f(x)h(x)$,
for some monic $h(x)\in K[x]$. Since $A$ is a PID, there is a $d\in A$
such that $dh(x)\in A[x]$ and $c(dh)=(1)$, where $c(\cdot)$ denotes
the content of a polynomial in $A[x]$. By Gauss's lemma, $d\cdot c(g)=c(f)c(dh)=(1)$,
so $d$ is a unit and $h(x)\in A[x]$. Thus $I=(f(x))$ and we have an isomorphism
\[
A[x]/(f(x))\cong A[\theta].
\]

The Latimer--MacDuffee correspondence is a bijection
between similarity classes of matrices in $\M_{n}(\Z)$ with characteristic polynomial $f(x)$ and ideal
classes of the order $\Z[\theta]$. It was first established by Latimer
and MacDuffee \cite{Latimer-MacDuffee} and later given a new simpler
proof by Taussky (see \cite{Taussky-Latimer-MacDuffee}, \cite[Chapter~III, Section~16]{Newman}
and the last appendix of \cite{Cohn-Taussky}). As shown by Bender
\cite{Bender} and Buccino \cite{Buccino}, Taussky's proof translates
readily to the more general situation of similarity classes of matrices
over an integral domain $R$ and classes of ideals that are free $R$-modules
of rank $n$. More precisely, Lemma~\ref{lem:Every ideal free rank n}
together with \cite[Theorem~1]{Bender} imply:
\begin{thm}
There is a canonical bijection between the similarity classes of matrices
in $\M_{n}(A)$ with characteristic polynomial $f(x)$ and ideal classes
in $A[\theta]$.
\end{thm}

The bijection is canonical in the following sense. Let $\alpha\in\M_{n}(A)$
have characteristic polynomial $f(x)$. Then $\theta$ is an eigenvalue of $\alpha$
and since $f(x)$ is irreducible and we have assumed that $L/K$ is separable, all the eigenvalues (in some algebraic closure $\bar{K}$ of $K$) are distinct
so the eigenspaces have dimension $1$ over $\bar{K}$
and thus
two eigenvectors corresponding to the same eigenvalue $\theta$ must
be proportional (by a factor in $\bar{K}$). Since $\alpha$ has an eigenvector 
in $K(\theta)^{n}$ corresponding to $\theta$, it also has an eigenvector $v=\begin{pmatrix}v_{1}\\
\vdots\\
v_{n}
\end{pmatrix}$ in $A[\theta]^{n}$ (by clearing denominators). Let $I=Av_{1}+\dots+Av_{n} \subseteq A[\theta]$.
Then $\alpha v=\theta v$ implies that $\theta I\subseteq I$ and
thus $\theta^{i}I\subseteq I$ for every $i\geq1$, which implies
that $I$ is an ideal of $A[\theta]$. It is then easy to see that
any other choice of eigenvector $v$ leads to an ideal that is in
the same ideal class as $I$ and that any other matrix in the similarity
class of $\alpha$ leads to the same ideal class (see \cite[Chapter~III, Section~16]{Newman}).
\begin{defn}
A matrix $\kappa\in\M_{n}(A)$ is called an \emph{ideal matrix} for
an ideal $\mfb$ of $A[\theta]$ (with respect to the basis $1,\theta,\dots,\theta^{n-1}$
of $A[\theta]$) if 
\[
\M_{n}(A)T(\mfb)=\M_{n}(A)\kappa.
\]
Such a $\kappa$ always exists because left ideals in $\M_{n}(A)$
are principal \cite[Theorem~II.5]{Newman}.

The following characterisation is taken as the definition of ideal
matrix in \cite{Taussky-Ideal_matrices_I} and \cite{Cohn-Taussky}.
The result is essentially a theorem of MacDuffee \cite{MacDuffee-1931},
which we here translate into our notation.
\end{defn}

\begin{lem}
\label{lem:equivalence-ideal-matrix}Let $\mfb$ be a non-zero ideal
of $A[\theta]$ and $\kappa\in\M_{n}(A)$. Then the following two conditions are equivalent.
\begin{enumerate}
    \item 
$\kappa$ is an ideal matrix
for $\mfb$ (with respect to the basis $1,\theta,\dots,\theta^{n-1}$
of $A[\theta]$), 
\item
the entries of
\[
\kappa\begin{pmatrix}1\\
\vdots\\
\theta^{n-1}
\end{pmatrix}
\]
 form an $A$-basis for $\mfb$.
 \end{enumerate}
\end{lem}

\begin{proof}
Let $a_{1},\dots,a_{n}$ be an $A$-basis for $\mfb$. In \cite{MacDuffee-1931},
the elements $\alpha_{1},\dots,\alpha_{k}$ are arbitrary generators
of the ideal. We may thus take $k=n$ and $\alpha_{i}=a_{i}$. Recall that
we have embedded $A[\theta]$ in $\M_{n}(A)$ via the regular embedding $T$.

Assume that the second condition holds
and write $\kappa=(k_{ij})$, $k_{ij}\in A$, so that 
\begin{equation}
a_{i}=\sum_{j=1}^{n}k_{ij}\theta^{j-1}.\label{eq:a_i-k_ij-theta}
\end{equation}
For any $1\leq i,j\leq n$, write
\begin{equation}
\theta^{i-1}\theta^{j-1}=\sum_{k=1}^{n}c_{ijk}\theta^{k-1},\label{eq:theta-i-theta-j-c_ijk}
\end{equation}
for some $c_{ijk}\in A$. We then have, for each fixed $j$, 
\begin{equation}
(c_{ijk})_{ik}=T(\theta^{j-1})\label{eq:c_ijk-theta}
\end{equation}
(note that $(c_{ijk})_{ik}=M_{j}$ in the notation of \cite{MacDuffee-1931}).
For any $r,i=1,\dots,n$ we then have
\[
\theta^{r-1}a_{i}=\sum_{j=1}^{n}k_{ij}\theta^{r-1}\theta^{j-1}=\sum_{j=1}^{n}k_{ij}\sum_{k=1}^{n}c_{rjk}\theta^{k-1}=\sum_{k=1}^{n}\Big(\sum_{j=1}^{n}k_{ij}c_{rjk}\Big)\theta^{k-1}.
\]
On the other hand, since $a_{1},\dots,a_{n}$ is a basis for $\mfb$,
there exist uniquely determined elements $q_{rij}\in A$ such that
\[
\theta^{r-1}a_{i}=\sum_{j=1}^{n}q_{rij}a_{j}=\sum_{j=1}^{n}q_{rij}\sum_{k=1}^{n}k_{jk}\theta^{k-1}=\sum_{k=1}^{n}\big(\sum_{j=1}^{n}q_{rij}k_{jk}\Big)\theta^{k-1}.
\]
From this we conclude that for any $r,i,k=1,\dots,n$, we have
\begin{equation}
\sum_{j=1}^{n}k_{ij}c_{rjk}=\sum_{j=1}^{n}q_{rij}k_{jk}.\label{eq:matrix-coeffs}
\end{equation}
Letting $Q_{i}=(q_{ris})_{rs}\in\M_{n}(A)$, for each $i$, the relations
(\ref{eq:a_i-k_ij-theta}) and (\ref{eq:c_ijk-theta}) imply that
(\ref{eq:matrix-coeffs}) can be rewritten as 
\[
T(a_{i})=Q_{i}\kappa.
\]
Thus $\kappa$ is a common right divisor of the matrices $T(a_{1}),\dots,T(a_{n})$.

Moreover, writing 
\[
a_{r}=\sum_{j=1}^{n}\delta_{rj}a_{j},
\]
for any $r=1,\dots,n$ (where $\delta_{ij}\in A$ is the Kronecker
delta) and 
\[
\delta_{rj}=\sum_{h=1}^{n}p_{rjh}\theta^{h-1},
\]
with $p_{rj1}=\delta_{rj}$ and $p_{rjh}=0$ if $h>1$, we get
\begin{align*}
\sum_{s=1}^{n}k_{rs}\theta^{s-1} & \overset{\eqref{eq:a_i-k_ij-theta}}{{=}}a_{r}=\sum_{j=1}^{n}\Big(\sum_{h=1}^{n}p_{rjh}\theta^{h-1}\Big)a_{j}=\sum_{j=1}^{n}\Big(\sum_{h=1}^{n}p_{rjh}\theta^{h-1}\Big)\Big(\sum_{l=1}^{n}k_{jl}\theta^{l-1}\Big)\\
 & \overset{\eqref{eq:theta-i-theta-j-c_ijk}}{=}\sum_{j,h,l=1}^{n}\sum_{s=1}^{n}p_{rjh}k_{jl}c_{hls}\theta^{s-1}.
\end{align*}
This implies that for each $r$ and $s$,
\[
k_{rs}=\sum_{j,h,l=1}^{n}p_{rjh}k_{jl}c_{hls},
\]
which, using (\ref{eq:a_i-k_ij-theta}) and (\ref{eq:c_ijk-theta}),
can be written
\[
\kappa=\sum_{j=1}^{n}P_{j}T(a_{j}),
\]
where $P_{j}=(p_{rjh})_{rh}$ (which equals the elementary matrix
with a $1$ in the $(j,1)$-position and has $0$s everywhere else).
Thus any common right divisor of $T(a_{1}),\dots,T(a_{n})$ in $\M_n(A)$ is a right divisor
of $\kappa$, so $\kappa$ is a greatest common right divisor and
therefore
\[
\M_{n}(A)T(\mfb)=\M_{n}(A)T(a_{1})+\dots+\M_{n}(A)T(a_{n})=\M_{n}(A)\kappa,
\]
that is, $\kappa$ is an ideal matrix for $\mfb$.

Conversely, assume that $\M_{n}(A)T(\mfb)=\M_{n}(A)\kappa$, so that
$\kappa$ is a greatest common right divisor of $T(a_{1}),\dots,T(a_{n})$.
Note that $\det(\kappa)\neq0$, or else we would have $\det(T(\mfb))=0$,
which implies that $\mfb=0$. This follows from the commutative diagram$$
\begin{tikzcd} 
L\arrow[hook,r] & \M_n(K)\\
A[\theta]\arrow[hook,r, "T"]\arrow[hook,u] & \M_n(A)\arrow[hook,u]
\end{tikzcd},
$$
where the top arrow is the regular embedding of $L$
with respect to the $K$-basis $1,\dots,\theta^{n-1}$; note that the image in $\M_{n}(K)$
of any non-zero element of $L$ is invertible. 

Let $\kappa'\in\M_{n}(A)$ be the (uniquely determined) matrix such
that 
\[
\kappa'\begin{pmatrix}1\\
\vdots\\
\theta^{n-1}
\end{pmatrix}=\begin{pmatrix}a_{1}\\
\vdots\\
a_{n}
\end{pmatrix}.
\]
We have shown above that $\kappa'$ must be a greatest common right
divisor of $T(a_{1}),\dots,T(a_{n})$, so, as $\kappa$ is invertible in
$\M_{n}(K)$, there exists an invertible matrix $\omega\in\GL_{n}(A)$
such that $\kappa=\omega\kappa'$. Thus the entries of 
\[
\kappa\begin{pmatrix}1\\
\vdots\\
\theta^{n-1}
\end{pmatrix}=\omega\begin{pmatrix}a_{1}\\
\vdots\\
a_{n}
\end{pmatrix}
\]
form an $A$-basis for $\mfb$.
\end{proof}
\begin{lem}
\label{lem:class contains ktk^-1}Let $\mfb$ be a non-zero ideal
of $A[\theta]$ and $\kappa$ an ideal matrix for $\mfb$ (with respect
to the basis $1,\theta,\dots,\theta^{n-1}$ of $A[\theta]$). The
similarity class in $\M_{n}(A)$ corresponding to the ideal class
of $\mfb$ under the Latimer--MacDuffee correspondence contains $\kappa T(\theta)\kappa^{-1}$.
\end{lem}

\begin{proof}
Let $\alpha\in\M_{n}(A)$ be a representative of the similarity class corresponding
to the ideal class of $\mfb$. Then 
\[
\alpha v=\theta v,
\]
where the entries $v_{1},\dots,v_{n}\in A[\theta]$
of $v$ span $\mfb$ over $A$. Since $\mfb$ is free of rank $n$
over $A$ (Lemma~\ref{lem:Every ideal free rank n}), the elements
$v_{1},\dots,v_{n}$ must be linearly independent, and thus form a
basis for the ideal $\mfb$. By Lemma~\ref{lem:equivalence-ideal-matrix}
there is an $\omega\in\GL_{n}(A)$ such that
\[
\omega\kappa\begin{pmatrix}1\\
\vdots\\
\theta^{n-1}
\end{pmatrix}=v,
\]
and therefore 
\[
(\omega\kappa)^{-1}\alpha\omega\kappa\begin{pmatrix}1\\
\vdots\\
\theta^{n-1}
\end{pmatrix}=\theta\begin{pmatrix}1\\
\vdots\\
\theta^{n-1}
\end{pmatrix}.
\]
This means that $(\omega\kappa)^{-1}\alpha\omega\kappa=T(\theta)$ and thus 
\[
\alpha=\omega\kappa T(\theta)\kappa^{-1}\omega^{-1},
\]
which lies in the same similarity class as $\kappa T(\theta)\kappa^{-1}$.
\end{proof}

\section{\label{sec:An-ideal-matrix for ideals of deg one}An ideal matrix
for ideals of degree one}
\begin{lem}
\label{lem:lambda-ideal-matrix}For any $z\in A$ the matrix
\[
\lambda=\begin{pNiceMatrix}[xdots/shorten=0.4cm]
-f(z) & 0 & \Cdots & \Cdots & 0\\
-z & 1 & \Ddots &  & \Vdots\\
-z^{2} & 0 & \Ddots & \Ddots & \Vdots\\
\vdots & \vdots & \Ddots & \Ddots & 0\\
-z^{n-1} & 0 & \Cdots & 0 & 1
\end{pNiceMatrix}
\]
is an ideal matrix for $(\theta-z)$, that is, $\M_{n}(A)T(\theta-z)=\M_{n}(A)\lambda$.
\end{lem}

\begin{proof}
We may perform successive elementary row operations on $T(\theta-z)$. The first operation consists of the row permutation  $(1,2,3,\dots,n)$:
\[T(\theta-z)=\begin{pNiceMatrix}[xdots/shorten=0.5cm]
-z & 1 & 0 & \Cdots & 0\\
0 & \Ddots[draw-first, shorten=0.3cm] & \Ddots & \Ddots & \vdots\\
\vdots & \Ddots & \Ddots & \Ddots & 0\\
0 & \Cdots \Cdots  & 0 & -z & 1\\
-k_{0} & -k_{1} & \Cdots & \Cdots & -k_{n}-z
\end{pNiceMatrix} \longmapsto \begin{pNiceMatrix}[xdots/shorten=0.5cm]
-k_{0} & -k_{1} & \Cdots & \Cdots & -k_{n}-z\\
-z & 1 & 0 & \Cdots & 0\\
0 & \Ddots[draw-first, shorten=0.3cm] & \Ddots & \Ddots & \vdots\\
\vdots & \Ddots & \Ddots & \Ddots & 0\\
0 & \Cdots \Cdots  & 0 & -z & 1
\end{pNiceMatrix}.
\]

Next, starting from the second column, successively use each of the
$1$s on the diagonal to clear the entry just below it. This yields
\begin{align*}
\begin{pNiceMatrix}[xdots/shorten=0.4cm]
-k_{0} & -k_{1} & \Cdots  & \Cdots \Cdots & \Cdots  & -k_{n}-z\\
-z & 1 & 0 & \Cdots & & 0\\
-z^{2} & 0 & 1 & \Ddots &  & \Vdots\\
0 & 0 & -z & 1 & \Ddots & \\
\vdots & \Ddots & \Ddots & \Ddots & \Ddots & 0\\
0 & \Cdots & 0  & 0 & -z & 1
\end{pNiceMatrix}\\
\longmapsto\begin{pNiceMatrix}[xdots/shorten=0.4cm] 
-k_{0} & -k_{1} & -k_{2} & \Cdots & \Cdots \Cdots & \Cdots & -k_{n}-z\\
-z & 1 & 0 & \Cdots &  & & 0\\
-z^{2} & 0 & 1 & \Ddots[draw-first] &  &  & \Vdots\\
-z^{3} & 0 & 0 & 1 &  & & \\
0 & 0 & 0 & -z & 1 &  & \\
\vdots & \Ddots & \Ddots & \Ddots & \Ddots & \Ddots & 0\\
0 & \Cdots & 0 & 0 & 0 & -z & 1
\end{pNiceMatrix}\\
\cdots\longmapsto\begin{pNiceMatrix}[xdots/shorten=0.4cm]
-k_{0} & -k_{1} & -k_{2} & \Cdots & -k_{n}-z\\
-z & 1 & 0 & \Cdots & 0\\
-z^{2} & 0 & \Ddots & \Ddots & \vdots\\
\vdots & \vdots & \Ddots & \Ddots & 0\\
-z^{n-1} & 0 & \Cdots & 0 & 1
\end{pNiceMatrix}.
\end{align*}
Finally, starting from the second column, successively use each of
the $1$s on the diagonal to clear the top row entry above it. The
resulting matrix is
\[
\lambda:=\begin{pNiceMatrix}[xdots/shorten=0.4cm]
g(z) & 0 & \Cdots & \Cdots & 0\\
-z & 1 & \Ddots &  &\\
-z^{2} & 0 & \Ddots & \Ddots[draw-first] & \Vdots\\
\vdots & \vdots & \Ddots & \Ddots & 0\\
-z^{n-1} & 0 & \Cdots & 0 & 1
\end{pNiceMatrix},
\]
for some polynomial $g(x)\in A[x]$. But since 
\begin{align*}
g(z) & =\det\lambda=(-1)^{\sgn(n,\dots,1)}\det(T(\theta-z))\\
 & =(-1)^{n-1}(-1)^{n}f(z)=-f(z),
\end{align*}
the result follows.
\end{proof}
For any $a,z\in A$, define the matrix
\begin{equation}
\kappa=\kappa(a,z)=\begin{pNiceMatrix}[xdots/shorten=0.4cm]
a & 0 & \Cdots &  & 0\\
-z & 1 & \Ddots &  & \Vdots\\
-z^{2} & 0 & \Ddots & \Ddots & \Vdots\\
\vdots & \vdots & \Ddots & \Ddots & 0\\
-z^{n-1} & 0 & \Cdots & 0 & 1
\end{pNiceMatrix} \label{eq:kappa}
\end{equation}

\begin{prop}
\label{prop:deg one - kappa ideal matrix}Assume that $\mfb$ is an
ideal of $A[\theta]$ of degree one and let $a$ be a generator of
$A\cap\mfb$. Then there exists a $z\in A$ such that $f(z)\equiv0\mod(a)$
and $\kappa=\kappa(a,z)$ is an ideal matrix for $\mfb$.
\end{prop}

\begin{proof}
By Lemma~\ref{lem:deg1-equivalence}, there exists a $z\in A$ such that $\mfb=(\theta-z,a)$.
Then $\theta\equiv z\mod\mfb$, so $f(z)\equiv f(\theta)\equiv0\mod\mfb$
and hence $f(z)\equiv0\mod(a)$, as $f(z)\in A$.
Thus
\[
\M_{n}(A)T(\mfb)=\M_{n}(A)T(\theta-z)+\M_{n}(A)a=\M_{n}(A)\lambda+\M_{n}(A)a,
\]
with $\lambda$ as in Lemma~\ref{lem:lambda-ideal-matrix}. Let $E_{ij}\in\M_{n}(A)$
be the matrix defined by having $(i,j)$-entry equal to $1$ and all
other entries equal to $0$. Writing $f(z)=am$, for $m\in A$, we
then have $\kappa=\lambda+E_{11}(1-m)a$ and hence
\[
\M_{n}(A)\kappa\subseteq\M_{n}(A)\lambda+\M_{n}(A)a.
\]
On the other hand, $aI\in\M_{n}(A)\kappa$ since the adjoint matrix
$\mathrm{adj}(\kappa)\in\M_{n}(A)$ satisfies $\mathrm{adj}(\kappa)\kappa=\det(\kappa)I=aI$.
Furthermore, $\lambda\in\M_{n}(A)\kappa$ since $\lambda=\kappa-E_{11}(1-m)a=(I-(1-m)E_{11}\mathrm{adj}(\kappa))\kappa$.
Thus $\M_{n}(A)\lambda+\M_{n}(A)a\subseteq\M_{n}(A)\kappa$ and so
$\M_{n}(A)T(\mfb)=\M_{n}(A)\lambda+\M_{n}(A)a=\M_{n}(A)\kappa$.
\end{proof}

\section{Representatives of similarity classes}
In the following lemma, note that $\kappa  =\kappa(a,z)$ in \eqref{eq:kappa} is not necessarily invertible in $\M_n(A)$ and that it is therefore not true in general that $\theta$ is similar (i.e., $\GL_n(A)$-conjugate) to $C_f(a,z)$.
\begin{lem}
\label{lem:C_f(a,z)} Let $a\in A$ and assume that there exists a
$z\in A$ such that $f(z)\equiv0\mod(a)$. Then the matrix $\kappa T(\theta)\kappa^{-1}\in\M_{n}(A)$
is similar to 
\[C_{f}(a,z)=\begin{pNiceMatrix}[xdots/shorten=0.4cm]
0 &  1 & 0 & \Cdots  & 0\\
\vdots & \Ddots &\Ddots & \Ddots& \vdots\\
0 & \Cdots &0 &1 & 0\\
u_{n-1} & u_{n-2} & \Cdots\Cdots & u_{1} & -f(z)a^{-1}\\
a & 0 & \Cdots \Cdots & 0 & z
\end{pNiceMatrix},
\]
where $u_{i}=-(z^{i}+k_{n-1}z^{i-1}+k_{n-2}z^{i-2}+\dots+k_{n-i+1}z+k_{n-i})$,
for $i=1,...,n-1$ with $k_i$ the coefficients of $f(x)$; see Section~\ref{sec:The-Latimer=002013MacDuffee-correspond}.
\end{lem}

\begin{proof}
Define the following matrices in $\GL_{n}(A)$:
\begin{align*}
\tau & =\begin{pNiceMatrix}[xdots/shorten=0.4cm] 
1 & 0 & \Cdots &  & 0\\
0 & & \Ddots &  & \Vdots\\
0 & -z & \Ddots[draw-first] &  & \\
\vdots & \Ddots & \Ddots & & 0\\
0 & \Cdots & 0 & -z & 1
\end{pNiceMatrix},\qquad v=\begin{pNiceMatrix}[xdots/shorten=0.4cm]
0 & 1 & 0 & \Cdots & 0\\
\Vdots & \Ddots[draw-first] & \Ddots & \Ddots & \vdots\\
 & &  &  & 0\\
0 & &  & & 1\\
1 & 0 & \Cdots & & 0
\end{pNiceMatrix}.
\end{align*}
Then direct computation shows that
\[
v\tau(\kappa T(\theta)\kappa^{-1})\tau^{-1}v^{-1}=C_{f}(a,z). \qedhere 
\]
\end{proof}
\begin{rem*}
Matrices of the form $C_{f}(a,z)$ were first introduced by Ochoa
(see \cite{Ochoa-I} and \cite{Ochoa-II}) who called them ``$A$-matrices'',
presumably because of the shape of their non-zero entries. The matrix
$C_{f}(a,z)$ is similar to the matrix
\[
\leftexp{\bf a}{\theta}=\begin{pNiceMatrix}[xdots/shorten=0.4cm]
0 &  & \Cdots & 0 & u_{n-1} & -f(z)a^{-1}\\
1 & \Ddots &  & & u_{n-2} & 0\\
0 & \Ddots & \Ddots & \Vdots & \Vdots & \Vdots\\
\Vdots & \Ddots & \Ddots & 0 &  & \Vdots\\
\Vdots &  & \Ddots & 1 & u_{1} & 0\\
0 & \Cdots & \Cdots & 0 & a & z
\end{pNiceMatrix}
\]
 considered by Rehm in \cite{Rehm-Ochoa}, associated with an ideal $\mathbf{a}$ (which we have called $\mathfrak{b}$). The difference is that we use a more standard form of the companion matrix $T(\theta)$ and corresponding variants of $\lambda$ and $\kappa$, compared to Rehm, in order to obtain matrices $C_{f}(a,z)$ that agree with Ochoa's $A$-matrices. To see that $\leftexp{\bf a}{\theta}$ is similar to $C_{f}(a,z)$, let 
\[
\sigma=\begin{pNiceMatrix}[xdots/shorten=0.4cm]
1 & 0 & & \Cdots & \Cdots & & 0\\
0 & \Ddots &  &  &  &  & \\
 & k_{n} &  & \Ddots &  &  & \Vdots\\
\Vdots & k_{n-1} &  & \Ddots & \Ddots &  & \Vdots\\
\Vdots & \Vdots & \Ddots & \Ddots &  & & \\
 &  &  & \Ddots & \Ddots &  & 0\\
0 & k_{3} & & \Cdots & k_{n-1} & k_{n} & 1
\end{pNiceMatrix},\qquad w=\begin{pNiceMatrix}[xdots/shorten=0.4cm] &  & 1\\
 & \Iddots\\
1
\end{pNiceMatrix}.
\]
One then easily verifies that 
\[
w\sigma(\kappa T(\theta)\kappa^{-1})\sigma^{-1}w^{-1}=\leftexp{\bf a}{\theta}.
\]
\end{rem*}
We can now finish the proof of our main result which gives conditions for similarity classes to contain matrices of the form $C_f(a,z)$ from  Lemma~\ref{lem:C_f(a,z)}. Recall that from
Section~\ref{sec:The-Latimer=002013MacDuffee-correspond} onwards
we assumed that $A$ is a PID. We have also assumed throughout that
$f(x)$ is irreducible.
\begin{thm} \label{thm:main}
Let $\mfb$ be an ideal of $A[\theta]$ of degree one. The similarity
class in $\M_{n}(A)$ that corresponds to the ideal class of $\mfb$
under the Latimer--MacDuffee correspondence contains a matrix of
the form $C_{f}(a,z)$, where $a,z\in A$ and $f(z)\equiv0\mod(a)$.
In particular, if $A[\theta]=B$ is the maximal order, then every
similarity class in $\M_{n}(A)$ contains a matrix of this form. 
\end{thm}

\begin{proof}
Let $\cC$ be a similarity class in $\M_{n}(A)$ corresponding to
the ideal class of $\mfb$. By Lemma~\ref{lem:class contains ktk^-1}
$\cC$ contains a matrix $\alpha T(\theta)\alpha^{-1}$, where $\alpha$
is an ideal matrix for $\mfb$. By Proposition~\ref{prop:deg one - kappa ideal matrix},
there exists a $z\in A$ such that $f(z)\equiv0\mod(a)$, where $a$
is a generator of $A\cap\mfb$ and $\kappa=\kappa(a,z)$ is an ideal
matrix for $\mfb$. Thus the similarity class of $\cC$ contains $\kappa T(\theta)\kappa^{-1}$,
so by Lemma~\ref{lem:C_f(a,z)} $\cC$ contains $C_{f}(a,z)$.

If $A[\theta]=B$, then by Corollary~\ref{cor:every ideal class contains deg one},
every ideal class contains an ideal of degree one, so the argument
above implies that every similarity class contains a matrix of the
desired form.
\end{proof}

\section{Examples} 
\begin{ex} 
    Let $A=\Z, K=\mathrm{Frac}(A)=\Q$, $L=\Q(\theta)$ be the number field with defining polynomial 
\[f(x)=x^3+4x-1\]
(LMFDB \cite{lmfdb} Number Field \href{https://www.lmfdb.org/NumberField/3.1.283.1}{\textcolor{blue}{3.1.283.1}}), and $B=\mathcal{O}_{L}=\Z[\theta]$ its ring of integers. $\Z[\theta]$ has ideal class group of order 2, generated by the ideal class of $\mathfrak{b}=(3,\theta-2)$. Let $T(\theta)$ be the companion matrix of $f(x)$, and $v$ be as in the proof of Lemma~\ref{lem:C_f(a,z)}:
\[T(\theta)=\begin{pmatrix}
	0 & 1 & 0\\ 
	0 & 0 & 1\\
	1 & -4 & 0\end{pmatrix}, \quad v=\begin{pmatrix}
	0 & 1 & 0\\
	0 & 0 & 1\\
	1 & 0 & 0 
\end{pmatrix}.\]
\begin{enumerate}
	\item For $[\mathfrak{b}]=[(3, \theta-2)]$, we may take $a=3, z=2$ to define $\kappa(3,2)$ (as defined in (\ref{eq:kappa})) as an ideal matrix for $\mathfrak{b}$, and $\tau_1$ as in the proof of Lemma~\ref{lem:C_f(a,z)}:
	\[\kappa_1=\kappa(3,2)=\begin{pmatrix} 
	3 & 0 & 0 \\
	-2 & 1 & 0\\
	-4 & 0 & 1
\end{pmatrix}, \quad \tau_1=\begin{pmatrix}
	1 & 0 & 0\\
	0 & 1 & 0\\
	0 & -2 & 1
\end{pmatrix}.\] 
By Lemma~\ref{lem:class contains ktk^-1}, the similarity class of $\M_n(\Z)$ corresponding to $\mathfrak{b}$ under the Latimer--MacDuffee correspondence contains $\kappa_1 T(\theta) \kappa_1^{-1}$, and a representative for this similarity class of the form $C_{f}(a,z)$ is given by:
\[v\tau_1(\kappa_1 T(\theta)\kappa_1^{-1})\tau_1^{-1}v^{-1}=\begin{pmatrix}
	0 & 1 & 0 \\
	-8 & -2 & -5\\
	3 & 0 & 2
\end{pmatrix}=C_{f}(3,2)\]
 
	\item For $[1]=[(1, \theta)]$, we may take $a=1, z=0$, so $\kappa_2=\tau_2=I$. The similarity class of $\M_n(\Z)$ corresponding to the trivial ideal class contains $T(\theta)$, and a representative for this similarity class of the form $C_f(a,z)$ is given by:
 \[v\tau_2(\kappa_2 T(\theta)\kappa_2^{-1})\tau_2^{-1}v^{-1}=vT(\theta)v^{-1}=\begin{pmatrix}
     0 & 1 & 0\\
     -4 & 0 & 1\\
     1 & 0 & 0
 \end{pmatrix}=C_f(1,0)\]
\end{enumerate}
\end{ex}

\begin{ex}
    Let $A=\F_2[x], K=\F_2(x)$, $L=\F_2(x,y)/(f(x,y))$ and $B=\F_2[x,y]/(f(x,y))$ be the (affine) coordinate ring of the curve 
    \[f(x,y)=y^3+x^3+x^2+x\] 
    over $\F_2$. Note that $f(x,y)$ is smooth, so $B$ is a Dedekind domain. 
    $B$ has ideal class group of order 3, generated by the ideal class of $\mathfrak{b}=(x,y)$. Let $T(y)$ be the companion matrix of $f(x,y)$, and $v$ be as in the proof of Lemma~\ref{lem:C_f(a,z)}:
\[T(y)=\begin{pmatrix}
	0 & 1 & 0\\
	0 & 0 & 1\\
	x^3+x^2+x & 0 & 0 
\end{pmatrix}, \quad v=\begin{pmatrix}
	0 & 1 & 0\\
	0 & 0 & 1\\
	1 & 0 & 0 
\end{pmatrix}.\]
\begin{enumerate}
	\item For $[\mathfrak{b}]=[(x,y)]$, we may take $a=x, z=0$ to define $\kappa(x,0)$ as an ideal matrix for $\mathfrak{b}$, and define $\tau_1$ as in the proof of Lemma~\ref{lem:C_f(a,z)}:
	\[\kappa_1=\kappa(x,0) =\begin{pmatrix}
		x & 0 & 0 \\
		0 & 1 & 0\\
		0 & 0 & 1
	\end{pmatrix}, \quad \tau_1=\begin{pmatrix}
		1 & 0 & 0\\
		0 & 1 & 0\\
		0 & 0 & 1
	\end{pmatrix}.
	\]
	The similarity class of $\M_n(\F_2[x])$ corresponding to $\mfp$ under the Latimer--MacDuffee correspondence contains $\kappa_1 T(y) \kappa_1^{-1}$, and a representative for this similarity class of the form $C_{f}(a,z)$ is given by:
	\[v\tau_1(\kappa_1 T(y) \kappa_1^{-1})\tau_1^{-1} v^{-1}=\begin{pmatrix}
	0 & 1 & 0 \\
	0 & 0 & x^2+x+1\\
	x & 0 & 0 
\end{pmatrix}
=C_f(x,0)\]
\item For $[\mathfrak{b}^2]=[(x+1,y+1)]$, we may take $a=x+1, z=1$ to define $\kappa(x+1,1)$ and $\tau_2$:
\[\kappa_2=\kappa(x-1, 1) =\begin{pmatrix}
	x+1 & 0 & 0 \\
	1 & 1 & 0\\
	1 & 0 & 1
\end{pmatrix}, \quad \tau_2=\begin{pmatrix}
	1 & 0 & 0\\
	0 & 1 & 0\\
	0 & 1 & 1
\end{pmatrix}.
\]
A representative for the similarity class of $\M_n(\F_2[x])$ containing $\kappa_2 T(y)\kappa_2^{-1}$ of the form $C_{f}(a,z)$ is given by:
\[v\tau_2(\kappa_2 T(y) \kappa_2^{-1})\tau_2^{-1} v^{-1}=\begin{pmatrix}
	0 & 1 & 0\\
	1 & 1 & x^2+1\\
	x+1 & 0 & 1
\end{pmatrix}
=C_f(x+1,1)
\]
\item For $[\mathfrak{b}^3]=[1]=[(1,y)]$, we may take $a=1,z=0$ so $\kappa_3=\kappa(1,0)=I$ and $\tau_2=I$. The similarity class of $\M_n(\F_2[x])$ corresponding to the trivial ideal class contains $T(y)$, and a representative for this similarity class of the form $C_{f}(a,z)$ is given by:
\[v\tau_3(\kappa_3 T(y)\kappa_3^{-1})\tau_3^{-1} v^{-1}=v T(y)v^{-1}=\begin{pmatrix}
0 & 1 & 0\\
0 & 0 & x^3+x^2+x\\
1 & 0 & 0
\end{pmatrix}=C_f(1,0).
\]
\end{enumerate}
\end{ex}

\begin{acknowledgement*}
We are deeply grateful to H.~W.~Lenstra for providing the proof
of Theorem~\ref{thm:Lenstra} and allowing us to include
it here.
\end{acknowledgement*}
\bibliographystyle{alex}
\bibliography{alex}

\end{document}